\documentclass[reqno]{amsart}  

\usepackage[ngerman,english]{babel}
\usepackage{amsmath}
\usepackage{amssymb}
\usepackage{amsthm}
\usepackage{latexsym}
\usepackage{enumitem}
	
\usepackage{graphicx}
\usepackage{mathtools}
\usepackage{microtype}
\usepackage[usenames,dvipsnames]{xcolor}
\usepackage[bookmarks, colorlinks=true, pdfstartview=FitV, linkcolor=Purple, citecolor=Purple, urlcolor=Purple]{hyperref}
\usepackage[font=small,labelfont=bf]{caption}
\usepackage{subcaption}

\newtheorem{theorem}{Theorem}
\newtheorem{lemma}[theorem]{Lemma}

\theoremstyle{definition}
\newtheorem{defn}[theorem]{Definition}

\numberwithin{equation}{section} 
\numberwithin{theorem}{section}
\numberwithin{figure}{section}

\begin{document}

\title[On the continuity of entropy of Lorenz maps]{On the continuity of entropy of Lorenz maps}

\author[Cooperband, Pearse, Quackenbush, Rowley, Samuel, West]{
Z.\,Cooperband, E.\,P.\,J.\,Pearse, B.\,Quackenbush, J.\,Rowley,
\\ \vspace{-1em}\\
T.\,Samuel and M.\,West}

\email{epearse@calpoly.edu, ajsamuel@calpoly.edu}

\address{\parbox[t]{0.935\textwidth}{Mathematics Department, California Polytechnic State University, CA, USA}}



\begin{abstract}
We consider a one parameter family of Lorenz maps indexed by their point of discontinuity $p$ and constructed from a pair of bilipschitz functions.  We prove that their topological entropies vary continuously as a function of $p$ and discuss Milnor's monotonicity conjecture in this setting.
\end{abstract}

\maketitle

\section{Introduction and main results}

Since the pioneering work of R\'enyi \cite{R:1957} and Parry \cite{MR0142719,MR0166332,Par,Par:1979}, an increasing amount of attention has been paid to maps of the unit interval.  Their study has provided solutions to practical problems within biology, engineering, information theory and physics.  Applications appear in analogue to digital conversion \cite{1011470}, analysis of electroencephalography (EEG) data \cite{Stolz:2017},  data storage \cite{LM:1995}, electronic circuits \cite{BV:2001}, mechanical systems with impacts and friction \cite{MR2015431} and relay systems \cite{MR2001701}.

The concept of topological entropy, now ubiquitous in the study of dynamical systems, was introduced by Adler, Konheim and McAndrews \cite{MR0175106} as a measure of the complexity of a dynamical system and is an invariant under a continuous change of coordinates, called topological conjugation.  Bowen \cite{MR0442989} gave a new, but equivalent, definition for a continuous map of a (not necessarily compact) metric space. For our purposes the following formulation, given by Misiurewicz and Szlenk in \cite{MR579440} and consistent with the definition given in \cite{MR0175106}, serves as a definition of the topological
entropy.  Let $T$ be a piecewise monotonic interval map, such as a Lorenz map (see Figure~\ref{fig:Lorenz}), the \textsl{topological entropy} $h(T)$ of $T$ is defined by
	\begin{align}\label{eq:entropy}
	h(T) \coloneqq \lim_{n \to \infty} \frac{1}{n} \ln (\operatorname{Var}(T^{n})),
	\end{align}
where $\operatorname{Var}(f)$ denotes the total variation of the function $f$.

The problem of comparing the topological entropies of two smooth interval maps which are close to each other, in a suitable sense, has been extensively studied in, for instance, \cite{MR1336987,MR3264762,MR1326374,MR1410784,MR1181083,MR1736945}.  This problem has also been studied in the setting of piecewise linear maps with one increasing branch and one decreasing branch, see for example \cite{MR3035348}.  We consider this problem for Lorenz maps, a class of interval maps with a single discontinuity and two increasing branches.  These maps play an important role in the study of the global dynamics of families of vector fields near homoclinic bifurcations, see \cite{MR556582,MR2178223,MR2217319,MR481632,MR1766514,MR1773551} and references therein.

Lorenz maps and their topological entropy have been and still are investigated intensely, see for instance \cite{BHV,MR3035348,MR1336987,MR3264762,G,GH,GS,H,HR,HS,MR1410784,LSS,MR1181083,MR1736945,SSV}. The simplest example of a Lorenz map is a (normalised) \mbox{$\beta$-transformation}, and the topological entropy of such a transformation is equal to $\ln(\beta)$; this was first shown in \cite{H,Par}. However, for a general Lorenz map the question of determining the topological entropy is much more complicated.  In \cite{G} Glendinning showed that every Lorenz map is semi-conjugate to an intermediate $\beta$-transformation and gave a criterion, in terms of kneading sequences, for when the semi-conjugacy is a conjugacy; see also \cite{MR3411531,BHV,DS:2004}. Note, this criterion turns out to be equivalent to topological transitivity.

\begin{figure}
\includegraphics[height=9.5em]{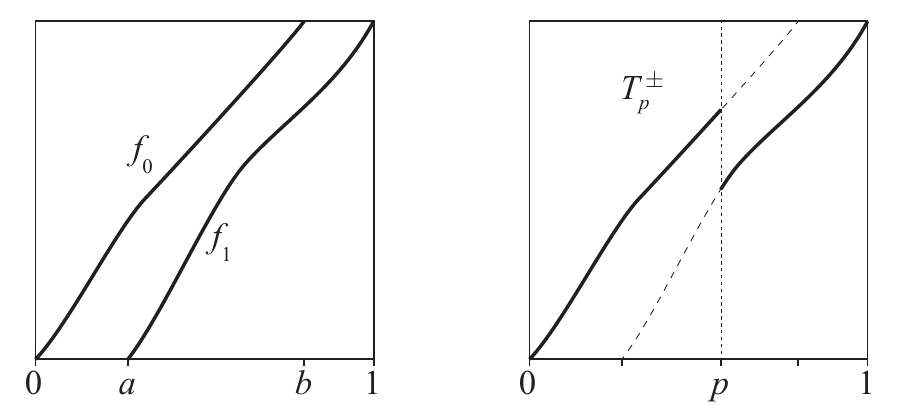}
\caption{A pair $(f_0, f_1)$ of branch functions and an associated Lorenz map.}
\label{fig:Lorenz}
\end{figure}

\begin{defn}\label{def:Lorenz}
Let $0 < a \leq p \leq b < 1$. An \textsl{upper}, or \textsl{lower}, \textsl{Lorenz map} is a map $T^+ \colon [0,1] \to [0, 1]$, respectively $T^- \colon [0,1] \to [0, 1]$, of the form 
	\begin{align*}
	T^+(x) \coloneqq \begin{cases}
	f_0(x) &\text{if} \; 0 \leq x < p, \\
	f_1(x) &\text{if} \; p \leq x \leq 1,
	\end{cases}
		\quad \text{respectively} \quad
	T^-(x) \coloneqq \begin{cases}
	f_0(x) &\text{if} \; 0 \leq x \leq p, \\
	f_1(x) &\text{if} \; p < x \leq 1.
	\end{cases}
	\end{align*}
where $f_0$ and $f_1$, called the \textsl{branch functions}, satisfy the following conditions.
	\begin{enumerate}[leftmargin=2.5em]
	\item\label{defn:part:Lorenz1} The functions $f_0 \colon [0,b] \to [0,1]$ and $f_1 \colon [a,1] \to [0,1]$ are continuous, strictly increasing and surjective.
	\item\label{defn:part:Lorenz2} There exist constants $C, c > 1$ with $C^{-1}\lvert x-y \rvert \leq \lvert f_i^{-1}(x) - f^{-1}_i(y) \rvert \leq c^{-1} \lvert x-y \rvert$ for $i \in \{ 0,1 \}$ and $x \in [0,1]$.
	\end{enumerate}
\end{defn}

When we wish to emphasis the point of discontinuity, we write $T^{\pm}_{p}$ for $T^{\pm}$.  Further, by definition, we have $h(T^{+}_{p}) = h(T^{-}_{p})$, and hence, for ease of notation, we let $h(T_{p})$ denote this common value.  Further,  for a fixed pair of branch functions, a direct consequence of \eqref{eq:entropy} is that
	\begin{align}\label{eq:entropy_lower_bound}
	h(T_{p}) \geq \ln (C)
	\end{align}
for all $p \in (a, b)$, where $a, b$ and $C$ are as in Definition~\ref{def:Lorenz}.

The main result of this article is the following.

\begin{theorem}\label{thm:continuity-of-entropy}
For a fixed pair of branch functions, $p \mapsto h(T_p)$ is continuous.
\end{theorem}

This paper is arranged as follows. In Section~\ref{sec:gensetup} we provide necessary definitions and preliminary results required for the proof of Theorem~\ref{thm:continuity-of-entropy}.  Section~\ref{sec:proof} is dedicated to the proof of Theorem~\ref{thm:continuity-of-entropy} and in Section~\ref{sec:affine_Lorenz} we include a discussion pertaining to Milnor's monotonicity conjecture in the setting of affine Lorenz maps.

\section{General setup}\label{sec:gensetup}

Throughout we use the convention that $\pm$ means either $+$ or $-$ and when we write $T^{\pm}_{p}$, we require that both $T_{p}^{+}$ and $T_{p}^{-}$ are defined using the same branch functions.

The set of all infinite words over the alphabet $\{0, 1 \}$ is denoted by $\Omega$ and is equipped with the discrete product topology.  For $n \in \mathbb{N}$, define $\Omega_{n}$ to be the set of finite words over the alphabet $\{ 0, 1\}$ of length $n$, and set $\Omega^{*} \coloneqq \bigcup_{n \in \mathbb{N}_{0}} \Omega_{n}$, where by convention $\Omega_{0}$ is the set containing only the \textsl{empty word} $\varnothing$.  For $\omega = \omega_{0} \cdots \omega_{k}$ and $v = v_{0} \cdots v_{n} \in \Omega^{*}$, we set $\omega v \coloneqq \omega_{0} \cdots \omega_{k} v_{0} \cdots v_{n}$, that is the \textsl{concatenation} of $\omega$ and $v$, and let $\overline{v} \coloneqq v v v \cdots$.  The \textsl{length} of $v \in \Omega^{*}$ is denoted by $\lvert v \rvert$ with $\lvert \varnothing \rvert=0$ and, 
for a natural number $k \leq \lvert v \rvert$, we set $v\lvert_{k} \coloneqq v_{0} \cdots v_{k-1}$.  We use the same notations when $v$ is an infinite word.

The continuous map $S \colon \Omega \to \Omega$ defined by $S( \omega_{0} \omega_{1} \cdots ) \coloneqq \omega_{1} \omega_{2} \cdots$ is called the \mbox{\textsl{left-shift}}.  We also allow for $S$ to act on finite words as follows. For $k \in \mathbb{N}_{0}$ and $v = v_{0} \cdots v_{k} \in \Omega^{*}$, we set $S(v) = v_{1} \cdots v_{k}$, if $k \geq 1$ and $S(v) = \varnothing$ otherwise.

The \textsl{upper} and \textsl{lower itinerary maps} $\tau_{p}^{\pm}  \colon [0,1] \to \Omega$ encode the orbit of a point $x \in [0, 1]$ under $T_p^\pm$ and are given by $\tau_{p}^{+}(x) = \omega_{0} \omega_{1} \cdots$ and $\tau_{p}^{-}(x) = v_{0} v_{1} \cdots$ where
	\begin{align*}
	\omega_{k} \coloneqq \begin{cases}
	0 & \text{if} \; (T^{+}_{p})^{k}(x) < p,\\
	1 & \text{if} \; (T^{+}_{p})^{k}(x) \geq p,
	\end{cases}
	\quad \text{and} \quad
	v_{k} \coloneqq 
	\begin{cases}
	0 & \text{if} \; (T^{-}_{p})^{k}(x) \leq p,\\
	1 & \text{if} \; (T^{-}_{p})^{k}(x) > p.
	\end{cases}
	\end{align*}
Here, for $n \in \mathbb{N}$, we denote by $(T^{\pm}_{p})^{n}$ the $n$-fold composition of $T^{\pm}_{p}$ with itself where $(T^{\pm}_{p})^{0}$ is set to be the identity map.  The infinite words $\alpha \coloneqq \tau_{p}^{-}(p)$ and $\beta \coloneqq \tau_{p}^{+}(p)$ are called the \textsl{kneading sequences} of $T^{\pm}_{p}$.  

We say that $\tau_{p}^{\pm}(p)$ is periodic if there exists $n \in \mathbb{N}$ such that $(T_{p}^{\pm})^{n}(p) = p$, and the \textsl{period} of $\tau_{p}^{\pm}(p)$ is the smallest $n \in \mathbb{N}$ for which this holds.  If $\tau_{p}^{\pm}(p)$ are periodic, then there exists $v, \omega \in \Omega^{*}$ such that $\tau_{p}^{+}(p) = \overline{v}$ and $\tau_{p}^{-}(p) = \overline{\omega}$.

\begin{lemma}[{\cite{BHV,BarnsleyMihalache}}]\label{thm:always-continuous}
The maps $x \mapsto \tau_x^{\pm}(x)$ are strictly increasing. Additionally, $x \mapsto \tau_x^+(x)$ is right-continuous and $x \mapsto \tau_x^-(x)$ is left-continuous.
\end{lemma}

The following lemma extends this result.

\begin{lemma}\label{thm:nonperiodic-to-continuous}
If $p \neq a$ and $\beta$ is non-periodic, then $x \mapsto \tau_x^+(x)$ is continuous at $p$ and if $p \neq b$ and $\alpha$ is non-periodic, then $x \mapsto \tau_x^-(x)$ is continuous at $p$.
\end{lemma}

\begin{proof}
We prove the first statement, as the proof of the second statement is identical.  Fix $\varepsilon>0$ and choose a natural number $N > 2$ with $2^{-N} < \varepsilon$. By definition and Lemma~\ref{thm:always-continuous}, it is sufficient to show there exists $\delta \in (0, p-a)$ such that $\tau_{p-\delta'}^+(p-\delta')\vert_N = \tau_p^+(p)\vert_N$ for all $\delta' \in (0, \delta)$.  This means we require $\delta >0$ so that for $n \in \{ 0, 1, \dots, N-1 \}$ and $\delta' \in (0, \delta)$ either 
	\begin{align}\label{eqn:both-on-the-same-side}
	\begin{aligned}
	(T_{p-\delta'}^+)^n(p-\delta') < &p-\delta' \; \text{and} \; (T_p^+)^n(p) < p  \\
	&\quad\text{or}\\
	(T_{p-\delta'}^+)^n(p-\delta') \geq &p-\delta' \; \text{and} \; (T_p^+)^n(p) > p.
	\end{aligned}
	\end{align}
To this end, let $c$ and $C$ be as in Definition~\ref{def:Lorenz} and choose $\delta \in (0, p-a)$ such that
	\begin{align}\label{eqn:delta_min_sqeeze}
	0 < \delta < \min \left\{ {\frac{(T^+_p)^k(p) - p}{C^k}} \colon k \in \{ 1, 2, \dots, N - 1 \} \right\}.
	\end{align}
Note, since $p \neq a$ and since $\beta$ is not periodic, the value on the left-hand-side of \eqref{eqn:delta_min_sqeeze} is positive.  We claim, for all $n \in \{ 0, 1, \dots, N -1 \}$ and $\delta' \in (0, \delta)$, that
	\begin{align}\label{eqn:inductive_close}
	\delta' c^n \leq (T_p^+)^n(p) - (T_{p-\delta'}^+)^n(p-\delta') \leq \delta' C^n.
	\end{align}
The base case, $n=0$, is immediate. Assume \eqref{eqn:inductive_close} holds for some $n \in \{ 0, 1, \dots, N-1 \}$. If $(T_p^+)^n(p) < p$, then \eqref{eqn:inductive_close} implies $(T_{p-\delta'}^+)^n(p-\delta') \leq (T_p^+)^n(p) -\delta' c^n < p-\delta'$.  If $(T_p^+)^n(p) \geq p$, then \eqref{eqn:inductive_close} implies $(T_p^+)^n(p) - p > \delta' C^n$; combining this with \eqref{eqn:delta_min_sqeeze} yields $(T_{p-\delta'}^+)^n(p-\delta') \geq (T_{p}^{+})^{n}(p) - \delta' C^{n} > p > p-\delta'$.  Therefore, by definition, we have \eqref{eqn:inductive_close} for $n+1$.  To complete the proof, notice that \eqref{eqn:inductive_close} implies \eqref{eqn:both-on-the-same-side}.
\end{proof}

In our proof of Theorem~\ref{thm:continuity-of-entropy}, we use the following Laurent series which can be thought of as a generating function of the kneading sequences $\alpha = \tau_{p}^{-}(p) = \alpha_{1} \alpha_{2} \cdots$ and $\beta = \tau_{p}^{+}(p) = \beta_{1} \beta_{2} \cdots$.  For $z \in \mathbb{C}\setminus\{0\}$, set
	\begin{align}\label{eqn:xi}
	\xi_{p}(z) \coloneqq \sum_{k=0}^{\infty}(\beta_{k} - \alpha_{k})z^{-k}.
	\end{align}
Observe that the interval $(1, 2)$ belongs to the domain of convergence of $\xi_{p}$. Further, we have the following result, which identifies the maximal zero of $\xi_{p}$ and the value $\gamma = \gamma_{p} \coloneqq \exp(h(T_{p}))$.

\begin{theorem}[{\cite{BHV,GH}}]\label{thm:entropy-is-the-zero}
The topological entropy of $T^{\pm}_p$ is equal to $\ln(r)$, where $r$ is the maximal positive real zero of $\xi_p$ in the interval $(1, 2]$.  Additionally, if the maximal zero of $\xi_p$ in the interval $(1, 2]$ is not simple, then this is the only zero of $\xi_p$ in the interval $(1, 2]$.
\end{theorem}

\section{Proof of Theorem~\ref{thm:continuity-of-entropy}}\label{sec:proof}

In Section~\ref{sec:nonperiodic} and Section~\ref{sec:periodic} for a fixed pair of branch functions we prove that the map $p \mapsto h(T_{p})$ is left-continuous; right-continuity follows by an identical argument, see Section~\ref{sec:right-continuity} for further details. The proof of left-continuity is subdivided into two \mbox{sub-cases}: when the kneading sequences of $T_{p}^{\pm}$ are not periodic, and when they are periodic.  For each sub-case we use the same approach.
	\begin{enumerate}[leftmargin=2.5em]
	\item\label{strategy_1} Fix $p \in (a, b)$ and $\varepsilon > 0$ with $(\gamma -\varepsilon, \gamma+\varepsilon) \subseteq (1,2)$.
	\item\label{strategy_2} Show there exists $\delta>0$ such that $\xi_{p-\delta}(x)$ has a maximal zero $r \in (\gamma-\varepsilon,\gamma+\varepsilon)$. 
	\item\label{strategy_3} Show there are no zeros larger than $r$.
	\end{enumerate}
With this at hand, Theorem~\ref{thm:entropy-is-the-zero} allows us to conclude that $\ln(r) = h(T_{p-\delta})$. Note, in Step (ii) we must take into account the multiplicity of $\gamma$; see Figure~\ref{fig:gamma-nbd}. If $\gamma$ has odd multiplicity, one can appeal to the intermediate value theorem, but more care is required in the case when $\gamma$ has even multiplicity.

\subsection{Case 1: $\beta$ non-periodic}\label{sec:nonperiodic}

Fix $\varepsilon > 0$ with $(\gamma -\varepsilon, \gamma+\varepsilon) \subseteq (1,2)$. By Lemma~\ref{thm:nonperiodic-to-continuous}, the assumption that $\beta$ is non-periodic ensures $x \mapsto \tau_x^{\pm}(x)$ are left-continuous at $p$.

To prove \eqref{strategy_2}, that is there exists $\delta>0$ such that $\xi_{p-\delta}(x)$ has a maximal zero $r \in (\gamma-\varepsilon,\gamma+\varepsilon)$, we replace the infinite sum $\xi_p(x)$ with a partial sum (a polynomial) and approximate $\gamma_{p-\delta}$ as the root of this polynomial. To this end, for $n \in \mathbb{N}$, set 
	\begin{align*}
	R_{p, n}(x) \coloneqq \sum_{k=n}^{\infty}(\beta_{k}-\alpha_{k})x^{-k}
	\quad \text{so that} \quad
	\xi_{p}(x) = \sum_{k=0}^{n-1}(\beta_{k}-\alpha_{k})x^{-k}+R_{p, n}(x).
	\end{align*}

\begin{lemma}\label{lem:claim1_non-periodic}
For $n \in \mathbb{N}$, there exists $\delta\,>\,0$ so that, for $\delta' \in (0, \delta)$ and $x \in (1,2)$,
	\begin{align}\label{eqn:claim-bound}
	\lvert \xi_{p-\delta'}(x) - \xi_p(x) \rvert \leq \frac{2x^{-n}}{1-x^{-1}}.
	\end{align}
\end{lemma}

\begin{proof}
For $m \in \mathbb{N}$, we have $\beta_{m}-\alpha_{m} \in \{-1, 0, 1\}$, whence, for $n \in \mathbb{N}$ and $x \in (1, 2)$,
	\begin{align}\label{eqn:Rp-bound}
	\lvert R_{p, n}(x) \rvert &\leq \sum_{k=n}^{\infty}x^{-k}=\frac{1}{1-x^{-1}}-\frac{1-x^{-n}}{1-x^{-1}}=\frac{x^{-n}}{1-x^{-1}}.
	\end{align}
Let $n \in \mathbb{N}$ be fixed. Since the maps $x \mapsto \tau_x^\pm(x)$ are both left-continuous at $p$, there exists $\delta > 0$ such that, if $\delta' \in (0, \delta)$, then $\tau_{p-\delta'}^{\pm}(p-\delta')\vert_{n} = \tau_{p}^{\pm}(p)\vert_{n}$.  As \eqref{eqn:Rp-bound} also holds for $R_{p-\delta', n}(x)$, we have established \eqref{eqn:claim-bound}.
\end{proof}

\begin{figure}
  \centering
  \begin{subfigure}[b]{0.475\linewidth}
    \centering\includegraphics[height=7em]{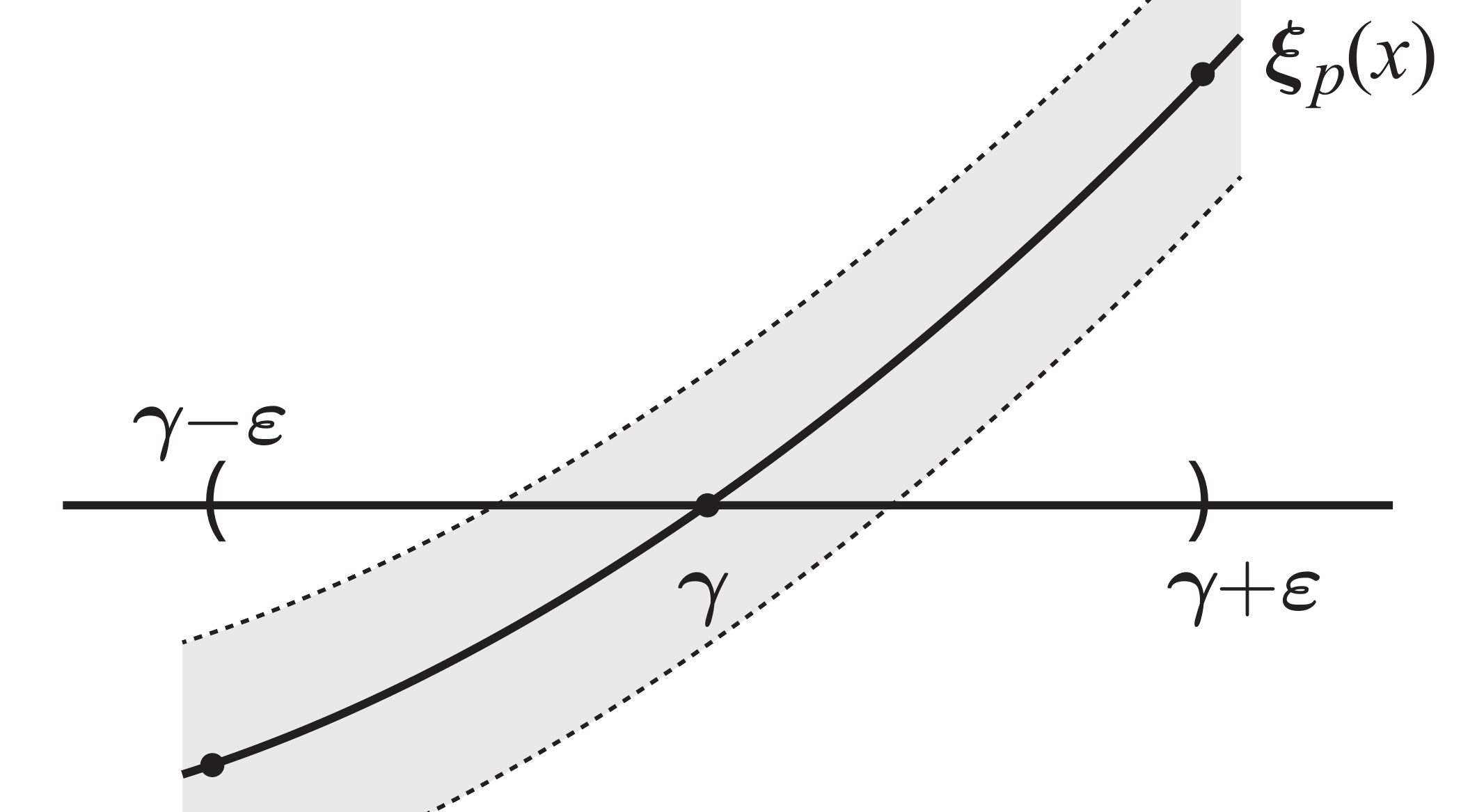}
    \subcaption{When $\gamma$ has odd multiplicity.}
  \end{subfigure}%
  \hspace{1em}
  \begin{subfigure}[b]{0.475\linewidth}
    \centering\raisebox{1.15em}{\includegraphics[height=6em]{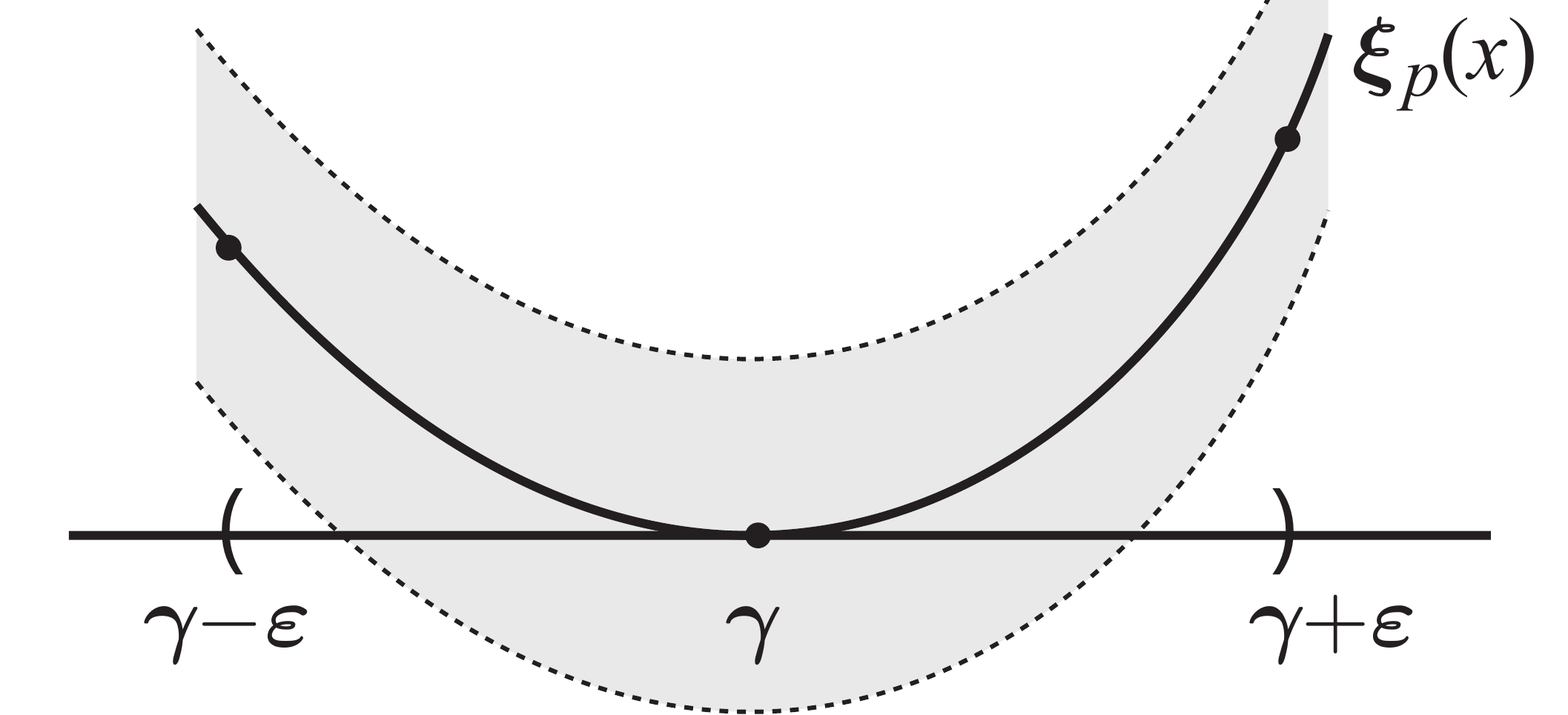}}
    \subcaption{When $\gamma$ has even multiplicity.}
  \end{subfigure}
  \caption{The graph of $\xi_p$ together with a neighbourhood in which the graph of $\xi_{p-\delta}$ belongs.}
  \label{fig:gamma-nbd}
\end{figure}

\begin{proof}[Proof of Theorem~\ref{thm:continuity-of-entropy}: left-continuity with $\beta$ non-periodic]
Note that $\gamma$ is an isolated zero. Indeed, in its domain of convergence, the function $\xi_{p}$ is holomorphic. Consequently, the existence of a sequence of zeros of $\xi_{p}$ converging to $\gamma$ would imply that $\xi_{p}$ is the constant zero function; a contradiction.  Let $\varepsilon > 0$ be fixed such that $(\gamma -\varepsilon, \gamma+\varepsilon) \subseteq (1,2)$ and such that $\xi_{p}$ has a single root in this interval, namely at $\gamma$.  Let $C > 1$ be as in Definition~\ref{def:Lorenz} and fix $u \in (1,C)$. The specific value of $u$ is not important, so for convenience set $u = (1+C)/2$, and by \eqref{eq:entropy_lower_bound} and Theorem~\ref{thm:entropy-is-the-zero} the real zeros of $\xi_{q}$ are greater than $u$ for $q \in (a, b)$.

Assume that $\gamma$ has odd multiplicity.  Let $n \in \mathbb{N}$ be such that 
	\begin{align*}
	\lvert \xi_{p}(\gamma - \varepsilon) \rvert \geq \frac{2u^{-n}}{1 - u^{-1}}
	\quad \text{and} \quad
	\lvert \xi_{p}(\gamma + \varepsilon) \rvert \geq \frac{2u^{-n}}{1 - u^{-1}},
	\end{align*}
and let $\delta$ be chosen in accordance with Lemma~\ref{lem:claim1_non-periodic}.  In which case, for all $\delta' \in (0,\delta)$,
	\begin{align*}
	\left \lvert \xi_{p-\delta'}(\gamma -\varepsilon) - \xi_{p}(\gamma -\varepsilon)\right \rvert
	< \frac{2u^{-n}}{1-u^{-1}} 
	\;\; \text{and} \;\;
	\left \lvert \xi_{p-\delta'}(\gamma +\varepsilon) - \xi_{p}(\gamma +\varepsilon)\right \rvert
	< \frac{2u^{-n}}{1-u^{-1}},
	\end{align*} 
which ensures $\operatorname{sgn}(\xi_{p-\delta'}(\gamma \pm \varepsilon)) = \operatorname{sgn}(\xi_{p}(\gamma \pm. \varepsilon))$, respectively.  This together with the fact that $\xi_{p}$ is smooth and has a single root in $(\gamma -\varepsilon, \gamma+\varepsilon)$ and an application of the intermediate value theorem yields that $\xi_{p_-\delta'}$ has a zero in $(\gamma -\varepsilon, \gamma+\varepsilon)$ for all $\delta' \in (0,\delta)$; see Figure~\ref{fig:gamma-nbd}.  

Assume that $\gamma$ has even multiplicity.  In this case, $\xi_p(x) > 0$ for all $x \neq \gamma$ with $x \in [u, 2]$, since $\xi_{p}$ is smooth, $\xi_{p}(2) > 0$ and, by Theorem~\ref{thm:entropy-is-the-zero}, we have $\xi_p$ has a single zero in the interval $(1, 2)$.  Let $\rho \coloneqq \inf \left\{ \xi_{p}(x)/2 \colon x \in [u, \gamma - \epsilon] \cup [\gamma + \epsilon, 2] \right\}$, let $n \in \mathbb{N}$ be such that 
	\begin{align*}
	\frac{2u^{-n}}{1 - u^{-1}} < \rho,
	\end{align*}
and let $\delta$ be chosen in accordance with Lemma~\ref{lem:claim1_non-periodic}.  In which case, for all $\delta' \in (0,\delta)$ and $x \in [u, \gamma - \epsilon] \cup [\gamma + \epsilon, 2]$,
	\begin{align*}
	\xi_{p-\delta'}(x) \geq \xi_{p}(x) - \frac{2u^{-n}}{1-u^{-1}} > \xi_{p}(x) - \rho > 0
	\end{align*} 
which ensures that $\xi_{p-\delta'}(x)$ has no zeros in $[u, \gamma - \epsilon] \cup [\gamma + \epsilon, 2]$.  Therefore, by Theorem~\ref{thm:entropy-is-the-zero} and \eqref{eq:entropy_lower_bound}, namely that the real zeros of $\xi_{q}$ belong to the interval $(u, 2]$ for $q \in (a, b)$, we have $\xi_{p_-\delta'}$ necessarily has a zero in the interval $(\gamma -\varepsilon, \gamma+\varepsilon)$

Therefore, regardless of the multiplicity of $\gamma$, it is necessarily the case that $\xi_{p_-\delta'}$ has a zero in $(\gamma -\varepsilon, \gamma+\varepsilon)$ for all $\delta' \in (0, \delta)$, whence Theorem~\ref{thm:entropy-is-the-zero} implies that $\lvert h(T_{p})-h(T_{p-\delta'}) \rvert \leq \varepsilon$ for all $\delta' \in (0, \delta)$, as required.
\end{proof}

\subsection{Case 2: $\beta$ periodic}\label{sec:periodic}

In this section we assume $\beta$ is periodic with period $N$, for some $N \in \mathbb{N}$. 

\begin{lemma}\label{lem:claim1}
For $n \in \mathbb{N}$ with $n \geq N$, there exists $\delta > 0$ such that, for all $\delta' \in (0, \delta)$, the concatenation of $\tau^+_{p}(p)\vert_{N}$ and $\tau^-_{p}(p)\vert_{n-N}$ is equal to $\tau^+_{p-\delta'}(p-\delta')\vert_n$, namely 
	\begin{align*}
	\tau^+_{p-\delta}(p-\delta')\vert_n = (\tau^+_{p}(p)\vert_{N}) (  \tau^-_{p}(p)\vert_{n-N}).
	\end{align*}
\end{lemma}

\begin{proof}
Let $n \geq N$ denote a fixed integer.  By Lemma~\ref{thm:always-continuous} there exists $\eta > 0$ such that, if $\eta' \in (0, \eta)$, then $(T_{p-\eta'}^+)^{j}(p-\eta') \neq p - \eta'$ for all $j \in \{ 1, 2, \dots, N+n-1 \}$.  Using the same arguments as in the proof of Lemma~\ref{thm:nonperiodic-to-continuous}, we may choose $\eta$ small enough so that, in addition to this, if $\eta' \in (0, \eta)$, then $\eta' c^j \leq (T_p^+)^j(p) - (T_{p-\eta'}^+)^j (p-\eta') \leq \delta C^j$ for all $j \in \{ 0, 1, \dots, N \}$; here $c$ and $C$ are as in Definition~\ref{def:Lorenz}.  If $j = N$, then this yields $\eta' c^N < p-(T_{p-\eta'}^+)^N (p-\eta') < \eta' C^N$ for $\eta' \in (0, \eta)$.  Since $c > 1$, this implies that $0 < (p-\eta') - (T_{p-\eta'}^+)^N(p-\eta') < \eta' (C^N-1)$.  By Lemma~\ref{thm:always-continuous}, there exists $\lambda > 0$ so that, for $q \in [a, b]$ with $0 < p - q < \lambda$, we have $\tau_{p}^-(p)\vert_n = \tau_{q}^-(q)\vert_n$.  Setting $\delta \coloneqq \min\{ \lambda C^{-N}/2, \eta \}$, if $\delta' \in (0, \delta)$, then $0 < (p-\delta') - (T_{p-\delta'}^+)^N (p-\delta') < \lambda/2$, which implies 
	\begin{align*}
	\tau_{p-\delta'}^-(p-\delta')\vert_n
	&= \tau_{p-\delta'}^-((T_{p-\delta'}^+)^N (p-\delta'))\vert_n\\
	&= \tau_{p-\delta'}^+((T_{p-\delta'}^+)^N (p-\delta'))\vert_n\\
	&= S^N(\tau_{p-\delta'}^+(p-\delta'))\vert_n.
	\end{align*}
Therefore, $\tau^+_{p-\delta}(p-\delta)\vert_n = (\tau^+_{p-\delta}(p-\delta)\vert_{N} ) ( \tau^-_{p-\delta}(p-\delta)\vert_{n-N})$. By the same arguments as in the proof of Lemma~\ref{thm:nonperiodic-to-continuous}, if $\delta > 0$ is suitably small, then $\tau^+_{p-\delta}(p-\delta)\vert_{N}=\tau^+_p(p)\vert_{N}$, and by Lemma~\ref{thm:always-continuous}, for $\delta > 0$ suitably small, $\tau^-_{p-\delta}(p-\delta)\vert_{n-N}=\tau^-_p(p)\vert_{n-N}$.
\end{proof}

\begin{lemma}\label{lem:claim2}
If $\varepsilon > 0$ and $(\gamma-\varepsilon,\gamma+\varepsilon) \subseteq (1,2)$, then there exists $n \in \mathbb{N}$ such that for the $\delta$ guaranteed by Lemma~\ref{lem:claim1}, we have 
\begin{align*}
	\lvert \xi_{p}(x) - \xi_{p-\delta'}(x) \rvert < \varepsilon, 
\end{align*}
for all $x \in (u, 2]$ and $\delta' \in (0, \delta)$, where, as in Section~\ref{sec:nonperiodic}, we set $u = (1+ C)/2$.
\end{lemma}

\begin{proof}
Since $\beta$ is periodic with period $N$, for all $x \in (1, 2]$,
	\begin{align*}
	\xi_{p}(x) 
	= \sum_{k=0}^{\infty} \beta_{k}x^{-k} - \sum_{k=0}^{\infty} \alpha_{k}x^{-k}
 	= \frac{1}{1-x^{-N}} \sum_{k=0}^{N-1} \beta_{k} x^{-k}  - \sum_{k=0}^{\infty} \alpha_{k}x^{-k}.
	\end{align*}
Let $v \in \mathbb{N}$ be fixed.  For $x \in (1, 2)$ set $\eta_{1, v}(x) \coloneqq \sum_{k=Nv}^\infty (\beta_k-\alpha_k) x^{-k}$ and observe
	\begin{align}\label{eqn:xi_p-decomposition2}
	\xi_{p}(x) 
 	= \frac{1-x^{-Nv}}{1-x^{-N}} \sum_{k=0}^{N-1} \beta_{k} x^{-k} 
		- \sum_{k=0}^{Nv-1} \alpha_{k} x^{-k} 
		+ \eta_{1, v}(x).
	\end{align}
Note, $\eta_{1, v}(x)$ is bounded by the tail of a geometric series, namely we have that $\lvert \eta_{1, v}(x) \rvert \leq x^{-Nv}/(1-x^{-1})$.

Let $n \geq N(v + 1)$.  By Lemma~\ref{lem:claim1}, for $\delta' \in (0, \delta)$, an expansion of $\xi_{p-\delta'}(x)$ similar to \eqref{eqn:xi_p-decomposition2} yields
	\begin{align}\label{eqn:xi_(p-d)-decomposition}
	\xi_{p-\delta'}(x)
	= \sum_{k=0}^{N-1} \beta_{k} x^{-k} 
		- \sum_{k=0}^{Nv-1} \alpha_{k} x^{-k}
		+ \sum_{k=N}^{Nv-1} \alpha_{k-N} x^{-k} 
		+ \eta_{2, v}^{(n)}(x, \delta')(x).
	\end{align}
Here $\eta_{2, v}^{(n)}(\cdot, \delta')$ consists of remaining terms in the expansion of $\xi_{p-\delta'}(x)$.  In \eqref{eqn:xi_(p-d)-decomposition} we have used the observation made in the last line of the proof of Lemma~\ref{lem:claim1}.  Also, note the remainder term $\eta_{2, v}^{(n)}(\cdot, \delta')$ is bounded by the tail of a geometric series, namely $\lvert \eta_{2, v}^{(n)}(x) \rvert \leq x^{-Nv}/(1-x^{-1})$.  Combining the above, we have
	\begin{align}\label{eqn:xi_p-xi_(p-d)}
	\!\!\xi_{p}(x)\!-\!\xi_{p-\delta'}(x)
	\!=\!\frac{x^{-N}}{1-x^{-N}} \sum_{k=0}^{N-1} \!\beta_{k} x^{-k} 
	\!- \eta_{3, v}(x) 
	- \!\sum_{k=N}^{Nv-1}\! \alpha_{k-N} x^{-k} 
	\!- \eta_{2,v}^{(n)}(x, \delta'),
	\end{align}
where $\eta_{3, v}(x) \coloneqq \sum_{k=Nv}^{\infty} \alpha_{k} x^{-k}$.  As with $\eta_{1, v}$ and $\eta_{2,v}^{(n)}(\cdot, \delta')$, observe that $\eta_{3, v}$ is bounded by the tail of a geometric series, namely $\lvert \eta_{3, v}(x) \rvert \leq x^{-Nv}/(1-x^{-1})$. Reindexing the second series in \eqref{eqn:xi_p-decomposition2} and rearranging yields 
	\begin{align*}
	\sum_{k=N}^{N(v+1)-1} \alpha_{k-N} x^{-k}
	&= x^{-N} \left(\frac{1-x^{-Nv}}{1-x^{-N}} \sum_{k=0}^{N-1} \beta_{k}x^{-k} 
		- \xi_{p}(x) + \eta_{1, v}(x)\right).
	\end{align*}
Substituting this into \eqref{eqn:xi_p-xi_(p-d)} gives, for $\delta' \in (0, \delta)$,
	\begin{align*}
	\begin{aligned}
	&\xi_{p}(x)-\xi_{p-\delta'}(x)\\
	&=  \frac{x^{-Nv} x^{-N}}{1-x^{-N}} \sum_{k=0}^{N-1} \beta_{k} x^{-k} 
	- \eta_{3, v}(x) 
	+ x^{-N} \xi_{p}(x) - x^{-N} \eta_{1, v}(x) 
	- \eta_{2, v}^{(n)}(x, \delta').
	\end{aligned}
	\end{align*}
Taking absolute values and using the bounds obtained for $\eta_{1, v}$, $\eta_{2, v}^{(n)}(\cdot, \delta')$ and $\eta_{3, v}$ gives, for a suitable constant $K>0$, that $\left\lvert (1-x^{-N})\xi_{p}(x)-\xi_{p-\delta'}(x) \right\rvert \leq K x^{-Nv}$, for all $\delta' \in (0, \delta)$ and $x \in (u, 2]$, as required.
\end{proof}

\begin{proof}[Proof of Theorem~\ref{thm:continuity-of-entropy}: left-continuity with $\beta$ periodic.]
The proof proceeds as in the proof of left-continuity of Theorem~\ref{thm:continuity-of-entropy} for $\beta$ non-periodic, with the following modification.  Instead of choosing $\delta$ in accordance with Lemma~\ref{lem:claim1_non-periodic}, we use Lemma~\ref{lem:claim1} and Lemma~\ref{lem:claim2} to choose $\delta$; the remainder of the proof follows identically.
\end{proof}

\subsection{Right continuity}
\label{sec:right-continuity}

In Section~\ref{sec:nonperiodic} and~\ref{sec:periodic}, the point $p$ was shifted to the left by subtracting some $\delta > 0$ from $p$. By adding $\delta>0$ to $p$ instead of subtracting it, right continuity can be shown by repeating Section~\ref{sec:nonperiodic} and~\ref{sec:periodic} with the following substitutions: swap $\tau_p^+$ with $\tau_p^-$, $T_p^+$ with $T_p^-$, and $p-\delta$ with $p+\delta$. Then, proceeding by cases (whether or not $\alpha$ is periodic). Therefore, it follows that $x \mapsto h(T_x)$ is continuous at every point $p \in (a, b)$. This concludes the proof of Theorem~\ref{thm:continuity-of-entropy}.

	\begin{figure}[t]
	\includegraphics[height=15.25em]{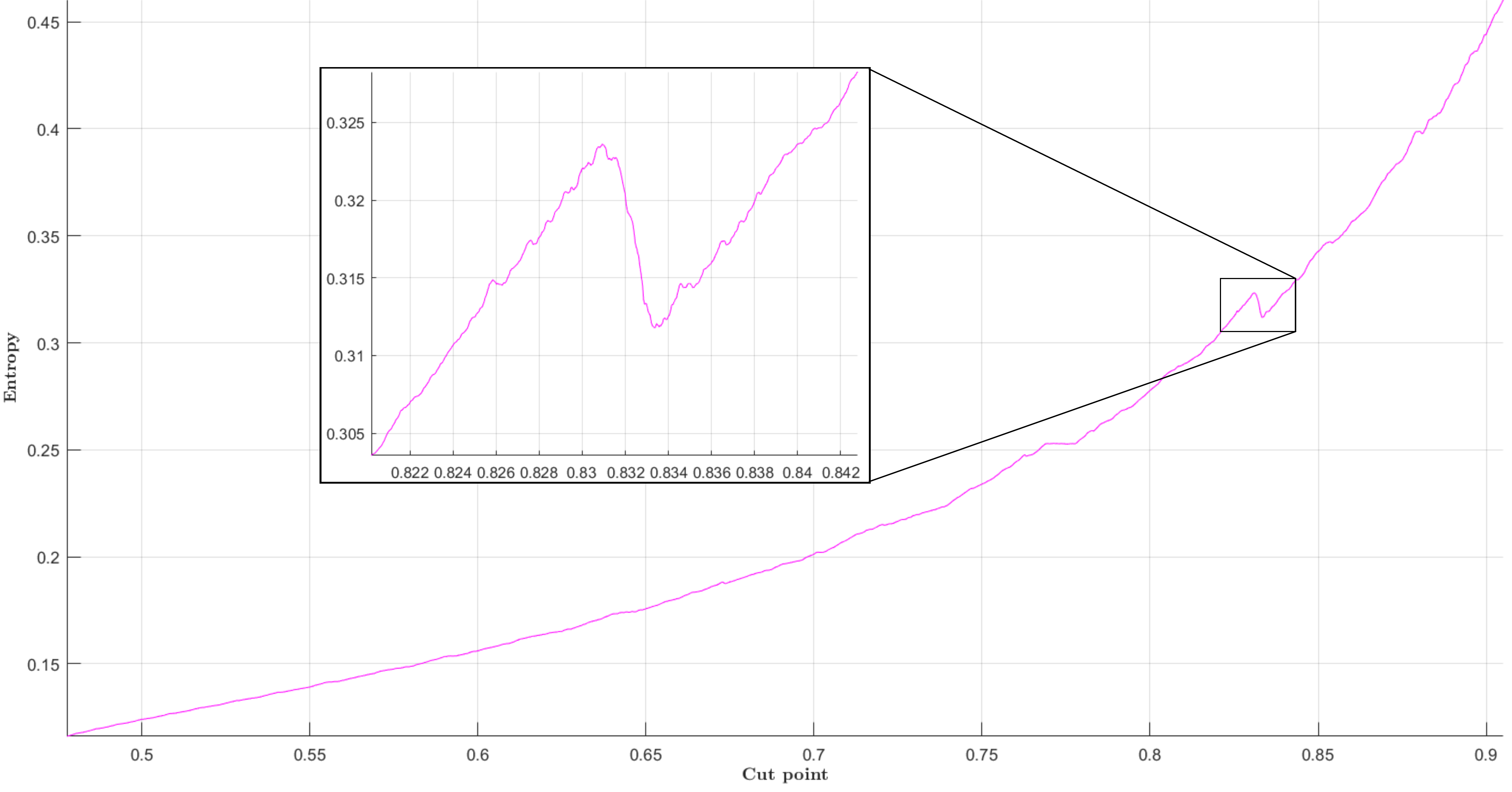}
	\caption{Numerical approximation of the topological entropy of the affine Lorenz map $T_{p}$, with first branch $f_{0}(x) \coloneqq 1.1 x$ and second branch $f_{1}(x) \coloneqq 1.9 x - 0.9$, using the algorithm developed in \cite{SSV} -- truncation term: $n=500$; tolerance: $\varepsilon=10^{-7}$.}
	\label{fig:affine-entropy}
	\includegraphics[height=15.25em]{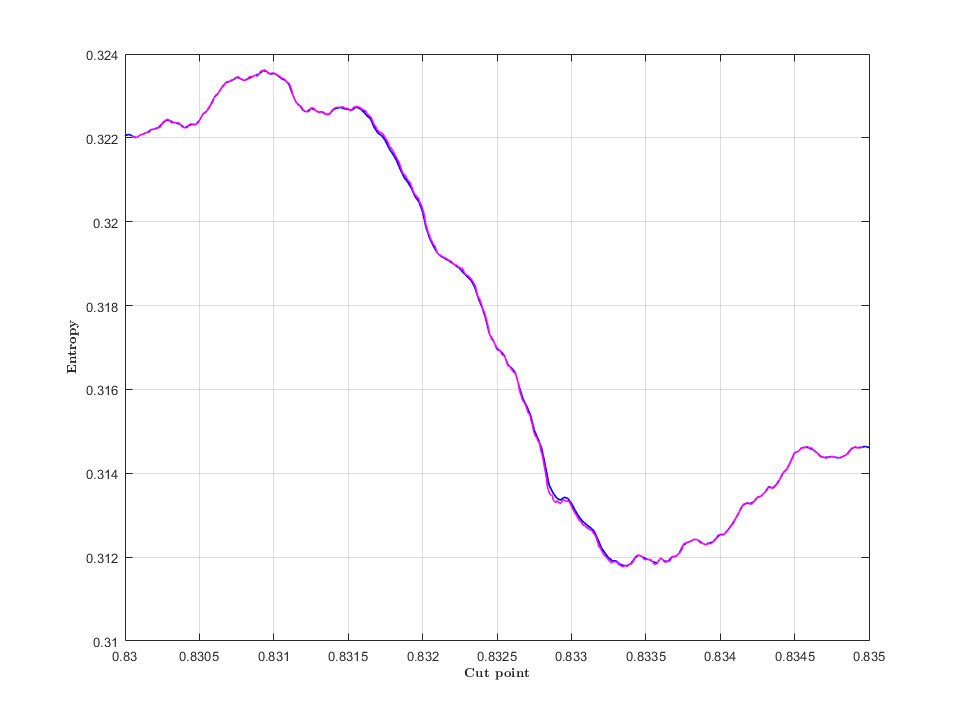}
	\caption{Agreement of separate numerical methods to compute topological entropy on the highlighted non-monotonic feature of Figure~\ref{fig:affine-entropy}.  The magenta curve uses the algorithm developed in \cite{SSV} and the blue curve computes the lap number of $(T_{p}^{\pm})^{50}$.}
	\label{fig:non-monotonic-entropy}
	\end{figure}

\section{Monotonicity of topological entropy}\label{sec:affine_Lorenz}

An \textsl{affine Lorenz map} $T_p$ is a Lorenz map where the branch functions $f_0$ and $f_1$ are affine, namely $f_0(x) = b_0 x$ and $f_{1} = 1 - b_{1} + b_{1}x$ where $0 < b_{0} \leq p \leq b_{1}$, $b_{0}, b_{1} > 1$ and $b_{0} + b _{1} > b_{0} b_{1}$.  In the special case when $b_0=b_1$, the map $T_p$ is called a \textsl{uniform Lorenz map} and is denoted by $T^{\pm}_{p,b}$.  Here $b$ denotes the common value $b_{0} = b_{1}$.

In \cite{H,Par} it is shown that $h(T_{p,b}) = \ln(b)$.  Further, Glendinning \cite{G}, Palmer \cite{P:1979} and Parry \cite{Par}, proved that a large class of piecewise monotone transformations of the unit interval are topologically conjugate to a uniform Lorenz map. 

\begin{theorem}[\cite{G,P:1979,Par}]
Every topologically transitive Lorenz map is topologically conjugate to a uniform Lorenz map.
\end{theorem}

In the following we take a first steps to address Milnor's monotonicity conjecture in the setting of affine Lorenz maps.  Indeed, numerical experiments show that there exist affine Lorenz maps $T_p^{\pm}(x)$ where $x \mapsto h(T_x)$ is non-monotonic and non-constant. This phenomena is shown using the algorithm developed in \cite{SSV} and is corroborated by a second algorithm that gives an approximation to topological entropy by computing lap numbers, see for instance \cite{MR1736945}.

The algorithm in \cite{SSV} computes topological entropy by comparing the kneading sequences of a given Lorenz map against kneading sequences of uniform Lorenz maps.  To verify its validity we compare the results against a separate computation which approximates the topological entropy of a given Lorenz map by computing the lap numbers of iterations of the original Lorenz map, see Figure~\ref{fig:non-monotonic-entropy}.

Here we considered the family affine Lorenz map $(T_{p})_{p \in [9/19, 10/11]}$ with branch functions $f_{0}$ and $f_{1}$ given by $f_{0}(x) \coloneqq 1.1 x$ and $f_{1}(x) \coloneqq 1.9 x - 0.9$.  The graph of the map $p \mapsto h(T_{p})$ is shown in Figure~\ref{fig:affine-entropy}; here $h(T_{p})$ has been computed by using the algorithm given in \cite{SSV} with truncation term $n=500$ and tolerance $\varepsilon=10^{-7}$.  We see that there are many instances of non-monotonicity in the plot that exceed the algorithm convergence tolerance.  Figure~\ref{fig:affine-entropy} captures a significant non-monotonic feature.  Indeed, there are many non-monotonic sections of the graph that exceed the algorithm error tolerance.

\section*{Acknowledgements}
	
The authors acknowledge the support of California Polytechnic State University's \textsl{Bill and Linda Frost Fund} and \textsl{College-Based Fees}. Part of this work was completed while T.~Samuel was visiting the Institut Mittag-Leffler as part of the research program \textsl{Fractal Geometry and Dynamics}. He is grateful to the organisers and staff for their very kind hospitality, financial support and a stimulating atmosphere.  
	
\bibliographystyle{plain}
\bibliography{entropy}
	
\end{document}